\documentclass[a4paper, 10pt]{article}
\usepackage[english]{babel}
\usepackage[nowrite]{frontespizio}

\usepackage{makeidx}
\usepackage{amssymb}
\usepackage{amsmath}
\usepackage{graphicx}
\usepackage{amsthm}
\usepackage{amsmath}
\usepackage{hyperref}
\hypersetup{colorlinks=true,linkcolor=blue}
\usepackage{layout}
\usepackage{latexsym}
\usepackage{float}
\usepackage{caption}
\usepackage{fullpage}
\usepackage{setspace}
\usepackage{color}
\usepackage{dcpic, pictex}

\def\beq{\begin{equation}}
\def\eeq{\end{equation}}
\def\beqn{\begin{eqnarray}}
\def\eeqn{\end{eqnarray}}

\def\supp{\mathop{\hbox{\rm supp}}}

\def\D{{\cal D}}

\def\N{\Bbb N}

\def\Ext{{\hbox{\rm Ext}}}
\newcommand {\no} {\noindent}

\newcommand{\R}{\mathbb{R}}

\newcommand{\G}{\mathcal{G}}
\newcommand{\La}{\mathcal{L}}
\newcommand{\de}{\mathrm {d}}
\newcommand{\Ncal}{\mathcal{N}^{p'(1-s)}_p}
\newcommand{\Ncaln}{\mathcal{N}^{p'(1-s)}_{p,n}}

\chardef\bslash=`\\ 

\newcommand{\dx}{\:{\rm d}x}

\newcommand{\dy}{\:{\rm d}y}
\newcommand{\dm}{\:{\rm d}\mu}
\newcommand{\estdom}{\mathcal{E}{\rm xt}}

\newtheorem{teo}{Theorem}[section]
\newtheorem{prop}{Proposition}[section]
\newtheorem{definition}{Definition}[section]

\newtheorem{lemma}{Lemma}[section]

\theoremstyle{definition}

\theoremstyle{remark}
\newtheorem{rem}{Remark}[section]

\begin{document}


\pagenumbering{arabic}

\title{Obstacle problems for the fractional $p$-Laplacian on fractal domains: well-posedness and asymptotics}

\author{\textsc{Simone Creo$^{(a)}$\footnote{Corresponding author. ORCID ID: 0000-0002-2083-2344} and Salvatore Fragapane$^{(b)}$}\\
\small{Dipartimento di Scienze di Base e Applicate per l'Ingegneria,}\\
\small{Sapienza  Universit\`{a} di Roma, Via A. Scarpa 16,  00161 Roma, Italy}\\
\small{email: $^{(a)}$ simone.creo@uniroma1.it;}\\
\small{$^{(b)}$ salvatore.fragapane5@scuola.istruzione.it, salvatore.fragapane@uniroma1.it}\\
} 
\date{}

 \maketitle
  \begin{abstract} 
	We study obstacle problems for the regional fractional $p$-Laplacian in a domain $\Omega\subset\R^2$ having as fractal boundary the Koch snowflake. We prove well-posedness results for the solution of the obstacle problem, as well as two equivalent formulations. Moreover, we study corresponding approximating obstacle problems in a sequence of domains $\Omega_n\subset\R^2$ having as boundary the $n$-th pre-fractal approximation of the Koch snowflake, for $n\in\N$. After proving the well-posedness of the approximating obstacle problems, we perform the asymptotic analysis for both $n\to+\infty$ and $p\to+\infty$.
\end{abstract}

\section*{Acknowledgments} The authors have been supported by the Gruppo Nazionale per l'Analisi Matematica, la Probabilit\`a e le loro Applicazioni (GNAMPA) of the Istituto Nazionale di Alta Matematica (INdAM). They also thank MUR for the support under the project PRIN 2022 -- 2022XZSAFN: \lq\lq Anomalous Phenomena on Regular and Irregular Domains: Approximating Complexity for the Applied Sciences" -- CUP B53D23009540006 (see the website \href{https://www.sbai.uniroma1.it/~mirko.dovidio/prinSite/index.html}{https://www.sbai.uniroma1.it/~mirko.dovidio/prinSite/index.html}).

\section*{Introduction}\label{Intro}
\setcounter{equation}{0}
In this paper we focus on obstacle problems involving fractional $p$-Laplace type operators in two-dimensional irregular domains having fractal or pre-fractal boundaries of Koch type.

There exists a huge literature voted to the study of problems for nonlocal, possibly nonlinear, fractional operators, and part of these works are done for operators acting in domains with fractal and pre-fractal boundaries. On the one hand, from a strictly mathematical point of view, these problems require attention when studying regularity issues; indeed, the presence of domains having such type of boundaries affects the regularity of the solutions and it requires the use of ad-hoc analytical tools, like suitable trace operators and functional spaces. Moreover, they give rise to so many interesting issues both purely theoretical, like uniqueness and asymptotic behavior, and more applied, like numerical analysis.\\
On the other hand, the consideration of these problems finds motivation in many questions connected to real problems. First of all, since the moment fractals were introduced (see \cite{MAN}) they have represented (and they still do) an innovation and an improvement in representing real objects; indeed, thanks to them, natural objects can be modeled in a more precise and accurate way, with respect to the possibilities provided by classical geometries. Moreover, fractals are a very powerful tool in the modeling of various physical phenomena. Actually, it is sufficient to think of the heat conduction phenomenon and remember how the surface plays a crucial role in it; in particular, as it is known, the larger the surface, the faster the conduction (see, for instance, \cite{CDL2012} and \cite{CreoCSF}). In this framework, pre-fractal approximating domains offer the possibility to increase the surface keeping the volume bounded. Furthermore, $p$-Laplace type operators are involved in the construction of mathematical models devoted to the studies of general physical, biological and engineering problems (see \cite{DIA} and the references quoted there).

Problems involving local or nonlocal $p$-Laplace type operators and/or domains with fractal boundary have been widely studied.\\
Naturally, in the analysis of $p$-Laplacian problems, issues as regularity, uniqueness and asymptotic behavior with respect to $p$ have played a crucial role and have been studied by many authors (see, for instance, \cite{MRT} and \cite{BDM} and the reference quoted therein). Moreover, problems related to fractal sets, such as their Hausdorff measure and their approximation by means of the corresponding pre-fractals, opened new perspectives (see for instance, \cite{HU}, \cite{mosco1} and the references quoted there). 
As far as we know, pioneering works in the study of problems involving both topics at the same time, that is $p$-Laplace type operators on domains with fractal and pre-fractal boundaries, are \cite{CV1} and \cite{MV}; subsequently, the results obtained there were used and further developed (see \cite{CPAA}, \cite{CFV}, \cite{F}, \cite{CF2} and the references quoted there). As to the study of problems involving nonlocal (possibly nonlinear and non-autonomous) operators in irregular domains, the literature is recent and goes back to the early 2020s; among the other, we refer to \cite{JEE}, \cite{JCA}, \cite{CLNODEA} and \cite{CLADE}.

Now, in the framework of problems involving the fractional $p$-Laplacian, different questions have been addressed. The issue of regularity, for instance, is analyzed in \cite{IMS}. Instead, speaking about obstacle problems for the fractional $p$-Laplacian in smooth domains, we refer to \cite{AHW}, where a variational obstacle problem with an exterior placed obstacle is studied and its equivalent formulations are proved. The issue of the asymptotic behavior with respect to $p$ was equally widely studied (see, for instance, \cite{BDD}, \cite{CLM}, \cite{FP} and \cite{dTEL}). Moreover, in \cite{MMV} an obstacle problem involving the infinity fractional Laplacian is considered.\\
It is worth pointing out that there is a deep link between the study of obstacle problems for nonlocal operators of fractional Laplacian type and free boundary problems. Among the others, we refer to \cite{rosoton}, \cite{BFRO} and \cite{CSS} and the references listed in.

In this paper, we consider, to our knowledge for the first time, obstacle problems for the regional fractional $p$-Laplacian in domains with fractal boundary of Koch type, with the aim of studying their asymptotic behavior from different points of view.\\
More precisely, we consider the following problems:
\begin{equation*}
\min_{v\in\mathcal{K}}J_p(v)\qquad\text{and}\qquad\min_{v\in\mathcal{K}^n}J_{p,n}(v),
\end{equation*}
where
\begin{equation*}
J_p(v)=\frac{1}{p}\iint_{\Omega\times\Omega}\frac{|v(x)-v(y)|^p}{|x-y|^{2+sp}}\dx\dy-\int_{\Omega}f(x)v(x)\dx+\frac{1}{p}\int_{\partial\Omega}b(x)|v(x)|^{p}\,\de\mu,
\end{equation*}
\begin{equation*}
\mathcal{K}=\{v\in W^{s,p}(\Omega)\,:\,\varphi_1\leq v\leq\varphi_2 \text{ in }\Omega\},
\end{equation*}
\begin{equation*}
J_{p,n}(v)=\frac{1}{p}\iint_{\Omega_n\times\Omega_n}\frac{|v(x)-v(y)|^p}{|x-y|^{2+sp}}\dx\dy-\int_{\Omega_n}f(x)v(x)\dx+\frac{\delta_n}{p}\int_{\partial\Omega_n}b(x)|v(x)|^{p}\,\de\ell,
\end{equation*}
\begin{equation*}
\mathcal{K}^n=\{v\in W^{s,p}(\Omega_n)\,:\,\varphi_{1,n}\leq v\leq\varphi_{2,n} \text{ in }\Omega_n\},
\end{equation*}
$f$ is a given function in a suitable Lebesgue space, $\varphi_1$, $\varphi_2$, $\varphi_{1,n}$ and $\varphi_{2,n}$ are given obstacles, $\delta_n$ is a fixed positive constant depending on $n$, $\Omega\subset\R^2$ is the bounded domain having as boundary the Koch snowflake and, for every $n\in\N$, $\Omega_n\subset\R^2$ is the bounded domain having as boundary the $n$-th approximation of the Koch snowflake (see Section \ref{geometria}).\\
Under suitable assumptions, our main goal is to study the asymptotic behavior of the solutions of the previous problems with respect to $p$ and $n$. In particular, for $p\to+\infty$, we prove that the solutions of the above problems converge to a solution of the following problems (see Theorems \ref{Ap} and \ref{Apn}):
\begin{equation*}
\max\left\{\int_{\Omega}f(x)v(x)\dx\,:\,v\in\mathcal{K}_{\infty}\right\}\qquad\text{and}\qquad\max\left\{\int_{\Omega_n}f(x)v(x)\dx\,:\,v\in\mathcal{K}^n_{\infty}\right\},
\end{equation*}
where
\begin{equation*}
\mathcal{K}_{\infty}=\{v\in W^{s,\infty}(\Omega)\,:\,\varphi_1\leq v\leq\varphi_2 \text{ in }\Omega \text{ with }\|v\|_{s,\infty}\leq1\}
\end{equation*}
and
\begin{equation*}
\mathcal{K}^n_{\infty}=\{v\in W^{s,\infty}(\Omega_n)\,:\,\varphi_{1,n}\leq v\leq\varphi_{2,n} \text{ in }\Omega_n\text{ with } \|v\|_{s,\infty,n}\leq1\}.
\end{equation*}
Moreover, for $p$ fixed (finite or infinite) and  $n\to+\infty$, we prove that the solution to the problem on $\Omega_n$ converges to a solution of the corresponding problem on $\Omega$ (see Theorems \ref{convergenza n p finito} and \ref{convergenza n p infinito}). Furthermore, we prove two equivalent formulations for the problem in the cases of $p$ fixed  (see Theorem \ref{Equi} and \ref{EquiN}) and discuss about the uniqueness (see Theorems \ref{Uni} and Theorems \ref{UniN}); in particular, we prove that the solutions of the obstacle problems (both in the fractal and pre-fractal cases) solve boundary value problems involving the regional fractional $p$-Laplacian with Robin boundary conditions. We point out that, to our knowledge, these are the first results of their kind in the setting of domains with fractal boundary. Moreover, we stress the fact that most of the mentioned bibliography do not consider Robin boundary conditions.\\
The organization of the paper is the following.\\
In Section \ref{preliminari} definitions and preliminary results and tools are given.\\
In Section \ref{SeP} we set the problem, proving equivalent formulations and uniqueness results.\\
Section \ref{AR} is devoted to the study of the asymptotic behavior both with respect to $p$ and to $n$.

\section{Preliminaries}\label{preliminari}
\setcounter{section}{1} \setcounter{equation}{0}

\subsection{Geometry}\label{geometria}

In this paper we denote points in $\R^2$ by $x=(x_1,x_2)$, the Euclidean distance by $|x-x_0|$ and the (open) Euclidean ball by $B(x_0,r)=\{x\in \R^2: |x-x_0|<r \}$ for $x_0\in \R^2$ and $r>0$. The Koch snowflake $K$ \cite{falconer} is the union of three coplanar Koch curves  $K^1$, $K^2$ and $K^3$, see Figure \ref{fig1}. 

\begin{figure}[H]
\centering
\includegraphics[width=0.2\textwidth]{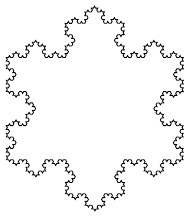}
\caption{The Koch snowflake $K$.}\label{fig1}
\end{figure}

\noindent The Hausdorff dimension of the Koch snowflake is $d_f=\frac{\ln 4}{\ln 3}$.\\ 
The natural finite Borel measure $\mu$ supported on $K$ is defined as
\begin{equation}\label{eq:1}
\mu:=\mu_1+\mu_2+\mu_3,
\end{equation}
where $\mu_i$ denotes the normalized $d_f$-dimensional Hausdorff measure, restricted to $K^i$, $i=1,2,3$.

We denote by
\begin{equation}\label{eq:3bitris}
K_{n+1}=\displaystyle\bigcup_{i=1}^3K^i_{n+1}
\end{equation}
\no the closed polygonal curve approximating $K$ at the $(n+1)$-th step. We denote by $K^i_{n+1}$ the pre-fractal (polygonal) curve approximating $K^i$.

The measure $\mu$ enjoys the following property:
\begin{equation}\label{eq:9bis}
\exists\,c_1,c_2>0\,:\,c_1r^{d_f}\leq \mu(B(P,r)\cap K)\leq c_2r^{d_f}\quad\forall\,P\in K.
\end{equation}
Since $\mu$ is supported on $K$, in \eqref{eq:9bis} we replace $\mu(B(P,r)\cap K)$ with $\mu(B(P,r))$.

Let $\Omega$ denote the two-dimensional open bounded domain with boundary $\partial\Omega=K$ and, for every $n\in\N$, let $\Omega_n$ be the pre-fractal polygonal domains approximating $\Omega$ at the $n$-th step, and let $K_n=\partial \Omega_n$ be the pre-fractal curves.
We note that the sequence $\{\Omega_n\}_{n \in \N}$ is an \emph{invading} sequence of sets exhausting $\Omega$.\\
We denote by $M$ and  by $\stackrel{\circ} M$ any segment of $K_n$ and the related open segment respectively. 
By $\ell$ we denote the natural arc-length coordinate on each segment of the polygonal curve $K_n$ and we introduce the coordinates $x_1=x_1(\ell)$, $x_2=x_2(\ell)$, on every  segment $M^{(j)}_n$ of $K_n$, $j=1,\dots,3\cdot 4^n$. By $\de\ell$ we denote the one-dimensional measure given by the arc-length $\ell$.

Throughout the paper, $C$ will denote possibly different positive constants. The dependence of such constants on some parameters will be given in parentheses or as subscripts.

\subsection{Functional spaces}

Let $\G$ (resp. $\mathcal{S}$) be an open (resp. closed) set of $\R^N$.
By $L^p(\G)$, for $p\geq1$, we denote the Lebesgue space with respect to the Lebesgue measure $\dx$, which will be left to the context whenever that does not create ambiguity. By $L^p(\partial\G)$ we denote the Lebesgue space on $\partial\G$ with respect to a Hausdorff measure $\mu$ supported on $\partial \G$. By $\D(\G)$ we denote the space of infinitely differentiable functions with compact support on $\G$. By $C(\mathcal{S})$ we denote the space of continuous functions on $\mathcal{S}$ and, for $\alpha\in(0,1)$, we denote by $C^{0,\alpha}(\mathcal{G})$ the set of H\"older continuous functions on $\mathcal{G}$ of exponent $\alpha$.\\
By $W^{s,p}(\G)$, for $0<s<1$, we denote the fractional Sobolev space of exponent $s$. Endowed with the following norm
\begin{equation*}
\|u\|^p_{W^{s,p}(\G)}=\|u\|^p_{L^p(\G)}+\iint_{\G\times\G} \frac{|u(x)-u(y)|^p}{|x-y|^{N+sp}}\dx\dy,
\end{equation*}
it becomes a Banach space. Moreover, we denote by $|u|_{W^{s,p}(\G)}$ the seminorm associated to $\|u\|_{W^{s,p}(\G)}$ and we define, for every $u,v\in W^{s,p}(\G)$,
\begin{equation}\notag
(u,v)_{N,s,p}:=\iint_{\G\times\G}|u(x)-u(y)|^{p-2}\frac{(u(x)-u(y))(v(x)-v(y))}{|x-y|^{N+sp}}\dx\dy.
\end{equation}
We recall that, for $p=\infty$, the space $W^{s,\infty}(\G)$ coincides with the space $C^{0,s}(\G)$ (see \cite[Section 8]{hitch}).

In the following we will denote by $|A|$ the Lebesgue measure of a subset $A\subset\R^N$.
For $f\in W^{s,p}(\G)$, we define the trace operator $\gamma_0$ as
\begin{equation}\notag 
\gamma_0f(x):=\lim_{r\to 0}{1\over|B(x,r)\cap\G|}\int_{B(x,r)\cap\G}f(y)\dy
\end{equation}
at every point $x\in \overline{\G}$ where the limit exists. The above limit exists at quasi every $x\in \overline{\G}$ with respect to the $(s,p)$-capacity (see Definition 2.2.4 and Theorem 6.2.1 page 159 in \cite{AdHei}). From now one we denote the trace operator simply by $f|_{\G}$; sometimes we will omit the trace symbol and the interpretation will be left to the context.

We introduce the notion of $d$-set following \cite{JoWa}.
\begin{definition}\label{dset}
A closed nonempty set $\mathcal{S}\subset\R^N$ is a $d$-set (for $0<d\leq N$) if there exist a Borel measure $\mu$ with $\supp\mu=\mathcal{S}$ and two positive constants $c_1$ and $c_2$ such that
\begin{equation}\label{defindset}
c_1r^{d}\leq \mu(B(x,r)\cap\mathcal{S})\leq c_2 r^{d}\quad\forall\,x \in\mathcal{S}.
\end{equation}
The measure $\mu$ is called $d$-measure.
\end{definition}

We recall the definition of Besov space specialized to our case. For generalities on Besov spaces, we refer to \cite{JoWa}.
\begin{definition}
Let $\mathcal{F}$ be a $d$-set with respect to a $d$-measure $\mu$ and $\gamma=s-\frac{N-d}{p}\in(0,1)$. ${B^{p,p}_\gamma(\mathcal{F})}$ is the space of functions for which the following norm is finite:
$$
\|u\|^p_{B^{p,p}_\gamma(\mathcal{F})}=\|u\|^p_{L^p(\mathcal{F})}+\iint_{|x-y|<1}\frac{|u(x)-u(y)|^p}{|x-y|^{d+\gamma p}}\,\de\mu(x)\,\de\mu(y).
$$
\end{definition}

We recall that, for every $n\in\N$, $\Omega$ and $\Omega_n$ are 2-sets, $K_n$ is a 1-set and the Koch snowflake $K$ is a $d_f$-set.

Let $p'$ be the H\"older conjugate exponent of $p$. In the following, we will denote the dual of the Besov space $B^{p,p}_\alpha(K)$ with $(B^{p,p}_\alpha(K))'$; we point out that this space coincides with the space $B^{p',p'}_{-\alpha}(K)$ (see \cite{JoWa2}).

We now state a trace theorem for functions in $W^{s,p}(\G)$, where $\G$ could be either $\Omega$ or $\Omega_n$. For the proof, we refer to \cite[Theorem 1, Chapter VII]{JoWa}.

\begin{prop}\label{teotraccia} Let $\frac{2-d}{p}<s<1$. $B^{p,p}_\gamma(\partial\G)$ is the trace space of $W^{s,p}(\G)$ in the following sense:
\begin{enumerate}
\item[(i)] $\gamma_0$ is a continuous linear operator from $W^{s,p}(\G)$ to $B^{p,p}_\gamma(\partial\G)$;
\item[(ii)] there exists a continuous linear operator $\Ext$ from $B^{p,p}_\gamma(\partial\G)$ to $W^{s,p}(\G)$ such that $\gamma_0\circ \Ext$ is the identity operator in $B^{p,p}_\gamma(\partial\G)$.
\end{enumerate}
\end{prop}

We point out that, if $\G=\Omega_n$, then its boundary is polygonal. Hence, the trace space of $W^{s,p}(\Omega_n)$ is $B^{p,p}_{s-\frac{1}{p}}(K_n)$, and the latter space coincides with $W^{s-\frac{1}{p},p}(K_n)$.



We recall that $\Omega$ and $\Omega_n$ are extension domains in the sense of the following theorem. For details, we refer to \cite[Theorem 1, page 103 and Theorem 3, page 155]{JoWa}.
\begin{teo}\label{teo estensione} Let $\G$  be either $\Omega$ or $\Omega_n$, for $n\in\N$, $0<s<1$ and $p>1$. There exists a linear extension operator $\mathcal{E}{\rm xt}\colon\,W^{s,p}(\G)\to\,W^{s,p}(\R^N)$ such that
\begin{equation}\label{R-3d}
\|\mathcal{E}{\rm xt}\, w\|^p_{W^{s,p}(\R^N)}\leq {\overline C_{s}} \|w\|^p_{W^{s,p}(\G)},
\end{equation}
where the constant ${\overline C_{s}}>0$ depends on $s$.
\end{teo}

We stress the fact that, if $\G=\Omega_n$ for $n\in\N$, the constant $\overline C_{s}$ does not depend on $n$.

We recall the following Sobolev embedding result for a class of extension domains. We refer the reader to \cite[Theorem 8.2]{hitch}.
\begin{teo}\label{immsob} Let $s\in(0,1)$ and $p\geq 1$ be such that $sp>N$. Let $\Omega\subseteq\R^N$ be a $W^{s,p}$-extension domain with no external cusps. Then $W^{s,p}(\Omega)$ is continuously embedded in $C^{0,\alpha}(\Omega)$ for $\alpha:=\frac{sp-N}{p}$, i.e. there exists a positive constant $C=C(N,s,p,\Omega)$ such that, for every $u\in W^{s,p}(\Omega)$,
\begin{equation}\label{imm sobolev}
\|u\|_{C^{0,\alpha}(\Omega)}\leq C\|u\|_{W^{s,p}(\Omega)}.
\end{equation}
\end{teo}

\section{The obstacle problems}\label{SeP}
\setcounter{section}{2} \setcounter{equation}{0}

\subsection{The regional fractional $p$-Laplacian}

Let $\Omega\subset\R^2$ and $\Omega_n\subset\R^2$ be the domains having the Koch curve and the $n$-th pre-fractal approximating curve as boundary, respectively, given in Section \ref{geometria}.\\
Let $s\in(0,1)$, $p>1$, and $\G$ be either $\Omega$ or $\Omega_n$. We introduce the space
\begin{equation}\notag
\La^{p-1}_s(\G):=\left\{u\colon\G\to\R\,\,\text{measurable}\,\,:\,\int_\G \frac{|u(x)|^{p-1}}{(1+|x|)^{2+sp}}\dx<\infty\right\}.
\end{equation}
We introduce a \lq\lq normalized" regional fractional $p$-Laplacian $(-\Delta_p)^s_\G$, for $x\in\G$: 
\begin{equation}\label{fracreglap}
\begin{split}
(-\Delta_p)^s_\G u(x) &={\rm P.V.}\int_\G |u(x)-u(y)|^{p-2}\frac{u(x)-u(y)}{|x-y|^{2+sp}}\dy\\[4mm]
&=\lim_{\varepsilon\to 0^+}\int_{\{y\in\G\,:\,|x-y|>\varepsilon\}} |u(x)-u(y)|^{p-2}\frac{u(x)-u(y)}{|x-y|^{2+sp}}\dy,
\end{split}
\end{equation}
provided that the limit exists, for every function $u\in \La^{p-1}_s(\G)$.\\
We point out that in the literature there are different versions of fractional $p$-Laplacians. The operator introduced in \eqref{fracreglap} coincides, up to a multiplicative constant, with the regional fractional $p$-Laplace operator introduced in \cite{warmaNODEA}. Here, we consider this \lq\lq normalized" version since it is the most suitable for considering the asymptotics for $p\to+\infty$.

We define the space
\begin{equation*}
V((-\Delta_p)_\Omega^s,\Omega):=\{u\in W^{s,p}(\Omega)\,:\,(-\Delta_p)_\Omega^s u\in L^{p'}(\Omega)\,\,\text{in the sense of distributions}\},
\end{equation*}
which is a Banach space equipped with the norm
\begin{equation}\notag
\|u\|_{V((-\Delta_p)_\Omega^s,\Omega)}:=\|u\|_{W^{s,p}(\Omega)}+\|(-\Delta_p)_\Omega^su\|_{L^{p'}(\Omega)}.
\end{equation}

We recall the following generalized Green formula for the regional fractional $p$-Laplacian on the domain with fractal boundary $\Omega$. For the proof, we refer to Theorem 2.2 in \cite{CLNODEA}.

\begin{teo}[Fractional Green formula]\label{greenf}
There exists a bounded linear operator $\Ncal$ from $V((-\Delta_p)_\Omega^s,\Omega)$ to $(B^{p,p}_{\alpha}(\partial\Omega))'$.\\
The following generalized Green formula holds for every $u\in V((-\Delta_p)_\Omega^s,\Omega)$ and $v\in W^{s,p}(\Omega)$:
\begin{equation}\label{fracgreen}
\begin{split}
&\left\langle \Ncal u,v\right\rangle_{(B^{p,p}_{\gamma}(\partial\Omega))', B^{p,p}_{\gamma}(\partial\Omega)}
=-\int_\Omega (-\Delta_p)_\Omega^s u\,v\,\dx\\[4mm]
&+\iint_{\Omega\times\Omega}|u(x)-u(y)|^{p-2}\frac{(u(x)-u(y))(v(x)-v(y))}{|x-y|^{2+sp}}\,\dx\dy,
\end{split}
\end{equation}
where $\gamma:=s-\frac{2-d_f}{p}>0$.
\end{teo}

Analogously, we can define the fractional normal derivative $\Ncaln$ on $K_n$ (as a linear and continuous operator acting on $W^{s-\frac{1}{p},p}(K_n)$ by means of a \lq\lq pre-fractal" fractional Green formula by proceeding as in \cite{CLNODEA}.

\subsection{Equivalent formulations}

From now on, we take $2<sp<p$. Moreover, we will assume
\begin{equation}\label{GenHyp}
\begin{cases}
f\in L^1(\Omega),\\
b\in C(\overline\Omega)\,:\, b(x)>0\,\,\forall\,x\in\overline\Omega.
\end{cases}
\end{equation}
Let us consider the minimum problems
\begin{equation}\label{Pp}
\min_{v\in\mathcal{K}}J_p(v)
\end{equation}
and 
\begin{equation}\label{Ppn}
\min_{v\in\mathcal{K}^n}J_{p,n}(v),
\end{equation}
where
\begin{equation}\label{Jp}
J_p(v)=\frac{1}{p}\iint_{\Omega\times\Omega}\frac{|v(x)-v(y)|^p}{|x-y|^{2+sp}}\dx\dy-\int_{\Omega}f(x)v(x)\dx+\frac{1}{p}\int_{\partial\Omega}b(x)|v(x)|^{p}\,\de\mu,
\end{equation}
\begin{equation}\label{K}
\mathcal{K}=\{v\in W^{s,p}(\Omega)\,:\,\varphi_1\leq v\leq\varphi_2 \text{ in }\Omega\},
\end{equation}
\begin{equation}\label{Jpn}
J_{p,n}(v)=\frac{1}{p}\iint_{\Omega_n\times\Omega_n}\frac{|v(x)-v(y)|^p}{|x-y|^{2+sp}}\dx\dy-\int_{\Omega_n}f(x)v(x)\dx+\frac{\delta_n}{p}\int_{\partial\Omega_n}b(x)|v(x)|^{p}\,\de\ell,
\end{equation}
where $\delta_n$ is a fixed positive constant depending on $n$,
\begin{equation}\label{Kn}
\mathcal{K}^n=\{v\in W^{s,p}(\Omega_n)\,:\,\varphi_{1,n}\leq v\leq\varphi_{2,n} \text{ in }\Omega_n\},
\end{equation}
and the obstacles $\varphi_1(x),\varphi_2(x),\varphi_{1,n}(x),\varphi_{2,n}(x)$ are given and they satisfy the following assumptions:
\begin{equation}\label{HypOb}
\begin{cases}
\varphi_1,\varphi_2\in W^{s,p}(\Omega),\\
\varphi_1\leq\varphi_2 \text{ in }\Omega,
\end{cases}
\end{equation}
and, for every $n\in\N$,
\begin{equation} \label{HypObn}
\begin{cases}
\varphi_{1,n},\varphi_{2,n}\in W^{s,p}(\Omega),\\
\varphi_{1,n}\leq\varphi_{2,n} \text{ in }\Omega.
\end{cases}
\end{equation}

We recall the following useful results, which follow by direct inspection.

\begin{prop}\label{Prop1}
The functionals $J_p(v)$ and $J_{p,n}(v)$, defined in \eqref{Jp} and \eqref{Jpn} respectively, are convex, weakly lower semi-continuous and coercive and the sets $\mathcal{K}$ and $\mathcal{K}^n$ are closed and convex.
\end{prop}

\begin{lemma}\label{Dis1}
Let $a,b\in\R^N$ and $p>1$. It holds that
\begin{equation}\label{ausiliare}
\frac{1}{p}(|a|^p-|b|^p)\geq|b|^{p-2}b(a-b).
\end{equation}
\end{lemma}

The following Lemma provides us with an equivalent norm on $W^{s,p}(\Omega)$. The same result can be proved also for $W^{s,p}(\Omega_n)$. 
\begin{lemma}\label{EquiNorm}
Let us consider $2<sp<p$ and let $b$ satisfy \eqref{GenHyp}. Then, for any $u\in W^{s,p}(\Omega)$, the following norm
\begin{equation}\label{bsp}
\|u\|_{b,s,p}:=\left(|u|^p_{W^{s,p}(\Omega)}+\int_{\partial\Omega}b|u|^{p}\dm\right)^{\frac{1}{p}}
\end{equation}
is equivalent to $\|u\|_{W^{s,p}(\Omega)}$.
\end{lemma}
\proof
The proof can be obtained by adapting the proof of \cite[Theorem 2.3]{warmaCPAA}, see also \cite[Proposition 2.8]{PW} and \cite[Theorem 1.5]{CLNODEA}.
\endproof

For $u,v\in W^{s,p}(\Omega)$, we introduce
\begin{equation}\label{ap}
a_p(u,v):=\iint_{\Omega\times\Omega}\frac{|u(x)-u(y)|^{p-2}[u(x)-u(y)][v(x)-v(y)]}{|x-y|^{2+sp}}\,\dx\dy+\int_{\partial\Omega}b(x)|u(x)|^{p-2}u(x)v(x)\,\dm,
\end{equation}
and, for every $n\in\N$,
\begin{equation}\label{apn}
a_{p,n}(u,v):=\iint_{\Omega_n\times\Omega_n}\frac{|u(x)-u(y)|^{p-2}[u(x)-u(y)][v(x)-v(y)]}{|x-y|^{2+sp}}\,\dx\,\dy+\delta_n\int_{\partial\Omega_n}b(x)|u(x)|^{p-2}u(x)v(x)\,\de\ell.
\end{equation}

We are ready to prove that the problems just introduced admit equivalent formulations. In particular, the following results hold.
\begin{teo}\label{Equi}
Let us assume $2<sp<p$, $\mathcal{K}\neq\emptyset$ and that assumptions \eqref{GenHyp} and \eqref{HypOb} hold. Then the following problems are equivalent:\\
(i) $u\in\mathcal{K}$ is solution to Problem \eqref{Pp};\\
(ii) $u\in\mathcal{K}$ solves the variational problem:
\begin{equation}\label{DVp}
\text{find } u\in\mathcal{K}\,\,:\,\, a_p(u,v-u)-\int_{\Omega}f(x)(v(x)-u(x))\dx\geq 0\quad\text{ for every $v\in\mathcal{K}$};\\
\end{equation} 
(iii) there exists $u\in\mathcal{K}$ solving 
\begin{equation}\label{StrongP}
\begin{cases}
(-\Delta_p)^s_\Omega u(x)=f(x) &\text{ for a.e. }x\in\Omega^1,\\[2mm]
(-\Delta_p)^s_\Omega u(x)=(-\Delta_p)^s_\Omega \varphi_1(x) &\text{ for a.e. }x\in\Omega^2,\\[2mm]
(-\Delta_p)^s_\Omega u(x)=(-\Delta_p)^s_\Omega \varphi_2(x) &\text{ for a.e. }x\in\Omega^3,\\[2mm]
\displaystyle\left\langle\Ncal u,\psi\right\rangle_{(B^{p,p}_\gamma(\partial\Omega))',B^{p,p}_\gamma(\partial\Omega)}+\int_{\partial\Omega}b|u|^{p-2}u\psi\dm=0\quad &\text{for every } \psi\in\mathcal{K},\\
\varphi_1\leq u\leq\varphi_2 &\text{in }\Omega,
\end{cases}
\end{equation}
with $\gamma:=s-\frac{2-d_f}{p}$ and
\begin{equation}\notag
\begin{split}
&\Omega^1:=\{x\in\Omega\,:\,\varphi_1(x)< u(x)<\varphi_2(x)\},\\
&\Omega^2:=\{x\in\Omega\,:\, u(x)=\varphi_1(x)\},\\
&\Omega^3:=\{x\in\Omega\,:\, u(x)=\varphi_2(x)\},
\end{split}
\end{equation}
hence $\Omega=\Omega^1\cup\Omega^2\cup\Omega^3$.
\end{teo}

\proof In order to prove our statement, let us start by proving that $(i)\Longleftrightarrow (ii)$.\\
$(ii)\Longrightarrow (i)$: if $u\in\mathcal{K}$ is solution to Problem \eqref{DVp}, then $u\in\mathcal{K}$ is solution to Problem \eqref{Pp}.\\
Let $u\in\mathcal{K}$ be a solution to Problem \eqref{DVp} and, for every $v\in\mathcal{K}$, let $w(x):=v(x)-u(x)$. We have:
\begin{equation}\notag
\begin{split}
0&\leq \iint_{\Omega\times\Omega}\frac{|u(x)-u(y)|^{p-2}[u(x)-u(y)][w(x)-w(y)]}{|x-y|^{2+sp}}\dx\dy-\int_{\Omega}f(x)w(x)\dx\\
&+\int_{\partial\Omega}b(x)|u(x)|^{p-2}u(x)w(x)\dx=-\int_{\Omega}f(x)(v(x)-u(x))\dx+\int_{\partial\Omega}b(x)|u(x)|^{p-2}u(x)(v(x)-u(x))\dm\\
&+\iint_{\Omega\times\Omega}\frac{|u(x)-u(y)|^{p-2}[u(x)-u(y)][(v(x)-v(y))-(u(x)-u(y))]}{|x-y|^{2+sp}}\dx\dy.
\end{split}
\end{equation}
Applying Lemma \ref{Dis1} to the last two terms in the right-hand side of the above inequality, we get
$$0\leq-\int_{\Omega}f(x)v(x)\dx+\int_{\Omega}f(x)u(x)\dx+\frac{1}{p}\int_{\partial\Omega}b(x)|v(x)|^p\dm-\frac{1}{p}\int_{\partial\Omega}b(x)|u(x)|^p\dm$$
$$-\frac{1}{p}\iint_{\Omega\times\Omega}\frac{|u(x)-u(y)|^p}{|x-y|^{2+sp}}\dx\dy+\frac{1}{p}\iint_{\Omega\times\Omega}\frac{|v(x)-v(y)|^p}{|x-y|^{2+sp}}\dx\dy\quad \forall\,v\in\mathcal{K}.$$
The above inequality implies that $J_p(u)\leq J_p(v)$ for every $u\in\mathcal{K}$, that is, $u$ is solution to Problem \eqref{Pp}.

$(i)\Longrightarrow (ii)$: if $u\in\mathcal{K}$ is solution to Problem \eqref{Pp}, then $u\in\mathcal{K}$ is solution to Problem \eqref{DVp}.\\
Let $u\in\mathcal{K}$ be a solution to Problem \eqref{Pp}, that is $J_p(u)\leq J_p(v)$ for every $v\in\mathcal{K}$. This is equivalent to say that (see \cite{T})
$$\frac{1}{\lambda}[J_p(u+\lambda(v-u))-J_p(u)]\geq 0\quad\forall\, v\in\mathcal{K},\,\forall\,\lambda\in(0,1].$$
This means that
\begin{equation}\notag
\begin{split}
0&\leq\frac{1}{p\lambda}\iint_{\Omega\times\Omega}\frac{|(u+\lambda(v-u))(x)-(u+\lambda(v-u))(y)|^p}{|x-y|^{2+sp}}\dx\,\dy-\frac{1}{\lambda}\int_{\Omega}f(x)[u(x)+\lambda(v(x)-u(x))]\dx\\[2mm]
&+\frac{1}{p\lambda}\int_{\partial\Omega}b(x)|u(x)+\lambda(v(x)-u(x))|^{p}\dm-\frac{1}{p\lambda}\iint_{\Omega\times\Omega}\frac{|u(x)-u(y)|^p}{|x-y|^{2+sp}}\dx\dy+\frac{1}{\lambda}\int_{\Omega}f(x)u(x)\dx\\[2mm]
&-\frac{1}{p\lambda}\int_{\partial\Omega}b(x)|u(x)|^{p}\dm=\frac{1}{p\lambda}\iint_{\Omega\times\Omega}\frac{|(u+\lambda(v-u))(x)-(u+\lambda(v-u))(y)|^p-|u(x)-u(y)|^p}{|x-y|^{2+sp}}\dx\dy\\[2mm]
&-\int_{\Omega}f(x)[v(x)-u(x)]\dx+\frac{1}{p\lambda}\int_{\partial\Omega}b(x)\{|u(x)+\lambda(v(x)-u(x))|^p-|u(x)|^p\}\dx\quad \forall\,v\in\mathcal{K},\,\forall\,\lambda\in(0,1].
\end{split}
\end{equation}
Passing to the limit as $\lambda\to 0^+$ in the above, we get:
$$0\leq\iint_{\Omega\times\Omega}\frac{|u(x)-u(y)|^{p-2}[u(x)-u(y)][(v(x)-v(y))-(u(x)-u(y))]}{|x-y|^{2+sp}}\dx\dy$$
$$-\int_{\Omega}f(x)[v(x)-u(x)]\dx+\int_{\partial\Omega}b(x)|u(x)|^{p-2}u(x)[v(x)-u(x)]\dx\quad\forall\,v\in\mathcal{K},$$
i.e., $u$ is a solution of \eqref{DVp}. Hence, the first equivalence is proved.

Now, to complete the proof, let us prove that $(iii)\Longleftrightarrow (ii)$.\\
$(iii)\Longrightarrow (ii)$: if $u\in\mathcal{K}$ is solution of Problem \eqref{StrongP}, then it is also solution to Problem \eqref{DVp}.\\
Let $u\in\mathcal{K}$ be solution of Problem \eqref{StrongP}, and let $v\in\mathcal{K}$ be arbitrary. Hence, setting $w:=v-u$, by using the integration by parts formula \eqref{fracgreen}, we get
\begin{equation}\notag
\begin{split}
a_p(u,w)&=\iint_{\Omega\times\Omega}\frac{|u(x)-u(y)|^{p-2}[u(x)-u(y)][w(x)-w(y)]}{|x-y|^{2+sp}}\dx\dy+\int_{\partial\Omega}b(x)|u(x)|^{p-2}u(x)w(x)\dm\\
&=\int_\Omega (-\Delta_p)^s_\Omega u\,w\dx+\left\langle\Ncal u,w\right\rangle_{(B^{p,p}_\gamma(\partial\Omega))',B^{p,p}_\gamma(\partial\Omega)}+\int_{\partial\Omega}b(x)|u(x)|^{p-2}u(x)w(x)\dm\\
&=\int_\Omega f(x)\,w(x)\dx,
\end{split}
\end{equation}
i.e., $u$ solves \eqref{DVp}.

$(ii)\Longrightarrow (iii)$: if $u\in\mathcal{K}$ is solution to Problem \eqref{DVp}, then it solves also Problem \eqref{StrongP}.\\
Let $u\in\mathcal{K}$ be a solution of the variational inequality \eqref{DVp}. This means that
\begin{equation}\label{sopra}
\begin{aligned}
&\iint_{\Omega\times\Omega}\frac{|u(x)-u(y)|^{p-2}[u(x)-u(y)][w(x)-w(y)]}{|x-y|^{2+sp}}\dx\dy-\int_{\Omega}f(x)w(x)\dx\\
&+\int_{\partial\Omega}b(x)|u(x)|^{p-2}u(x)w(x)\dm\geq0 \text{ for every $v\in\mathcal{K}$, with } w(x)=v(x)-u(x).
\end{aligned}
\end{equation}
We adapt to our setting the techniques in the proof of \cite[Proposition 3.4]{AHW}. Let $$\Omega^1:=\left\{x\in\Omega\,:\,\varphi_1(x)<u(x)<\varphi_2(x)\right\}.$$ We take $\zeta\in\D(\Omega^1)$; then, by construction, it holds for a sufficiently small $\epsilon>0$ that $v:=u+\epsilon\zeta\in\mathcal{K}$.\\
Hence, by taking $v=u+\epsilon\zeta$ in \eqref{sopra} and integrating by parts using \eqref{fracgreen}, we obtain
\begin{equation}\notag
\epsilon\int_{\Omega}(-\Delta_p)^s_\Omega u(x)\zeta(x)\dx-\epsilon\int_{\Omega}f(x)\zeta(x)\dx\geq0\quad\forall\,\zeta\in\D(\Omega^1).
\end{equation}
If we take $v=u-\epsilon\zeta$ we obtain the opposite inequality. Hence, we get that
\begin{equation}\label{equazione a.e.}
(-\Delta_p)^s_\Omega u = f\quad\text{in }\D'(\Omega^1),
\end{equation}
and, by density, in particular in $L^{2}(\Omega^1)$, so it holds almost everywhere in $\Omega^1$.\\
For the cases $x\in\Omega^2$ and $x\in\Omega^3$, the thesis follows from the following general Lewy-Stampacchia inequality: if $u$ solves the variational inequality \eqref{DVp}, then
\begin{equation}\label{LS}
(-\Delta_p)^s_\Omega\varphi_2\vee f\leq (-\Delta_p)^s_\Omega u\leq (-\Delta_p)^s_\Omega\varphi_1\wedge f,
\end{equation}
where $u\wedge v=\inf\{u,v\}$ and $u\vee v=\sup\{u,v\}$. We refer to \cite[Theorem 2.4 and Remarks 3.3 and 3.4]{giglimosconi}, for more details, see also \cite{IMS}.

Let now $\psi\in\mathcal{K}$. Then, for a sufficiently small $\epsilon>0$, in $\Omega^1$ it holds that $v:=u+\epsilon\psi\in\mathcal{K}$. Hence, the variational inequality \eqref{DVp}, together with the fractional Green formula \eqref{fracgreen} and \eqref{equazione a.e.}, yield
\begin{equation}\notag
\epsilon\int_{\partial\Omega} b(x)|u(x)|^{p-2}u(x)\psi(x)\dm+\left\langle\Ncal u,\epsilon\psi\right\rangle_{(B^{p,p}_\gamma(\partial\Omega))',B^{p,p}_\gamma(\partial\Omega)}\geq 0.
\end{equation}
Since, if we take $v=u-\epsilon\psi$, we obtain the opposite inequality, it holds that
\begin{equation}\notag
\int_{\partial\Omega} b(x)|u(x)|^{p-2}u(x)\psi(x)\dm+\left\langle\Ncal u,\psi\right\rangle_{(B^{p,p}_\gamma(\partial\Omega))',B^{p,p}_\gamma(\partial\Omega)}=0\quad\text{for every }\psi\in\mathcal{K}.
\end{equation}

\endproof

\begin{teo}\label{EquiN}
Let us assume $2<sp<p$, $\mathcal{K}^n\neq\emptyset$ for every $n\in\N$, $\delta_n>0$ and that assumptions \eqref{GenHyp} and \eqref{HypObn} hold. Then, for every $n\in\N$, the following problems are equivalent:\\
(i) $u_n\in\mathcal{K}^n$ is solution to Problem \eqref{Ppn};\\
(ii) $u_n\in\mathcal{K}^n$ solves the variational problem:
\begin{equation}\label{DVpn}
\text{find }  u_n\in\mathcal{K}^n\,\,:\,\,a_{p,n}(u_n,v_n-u_n)-\int_{\Omega_n}f(x)(v_n(x)-u_n(x))\dx\geq 0\quad \text{ for every  $v_n\in\mathcal{K}^n$};
\end{equation}
(iii) there exists $u_n\in \mathcal{K}^n$ solving 
\begin{equation}\label{StrongPn}
\begin{cases}
(-\Delta_p)^s_{\Omega_n} u_n(x)=f(x) &\text{ for a.e. }x\in\Omega^1_n,\\[2mm]
(-\Delta_p)^s_{\Omega_n} u_n(x)=(-\Delta_p)^s_\Omega \varphi_{1,n}(x) &\text{ for a.e. }x\in\Omega^2_n,\\[2mm]
(-\Delta_p)^s_{\Omega_n} u_n(x)=(-\Delta_p)^s_\Omega \varphi_{2,n}(x) &\text{ for a.e. }x\in\Omega^3_n,\\[2mm]
\displaystyle\left\langle\Ncaln u_n,\psi\right\rangle_{(W^{s-\frac{1}{p},p}(\partial\Omega_n))',W^{s-\frac{1}{p},p}(\partial\Omega_n)}+\delta_n\int_{\partial\Omega_n}b|u_n|^{p-2}u_n\psi\,\de\ell=0\quad &\text{for every }\psi\in\mathcal{K}^n,\\[2mm]
\varphi_{1,n}\leq u_n\leq\varphi_{2,n} &\text{in }\Omega_n,
\end{cases}
\end{equation}
with
\begin{equation}\notag
\begin{split}
&\Omega^1_n:=\{x\in\Omega_n\,:\,\varphi_{1,n}(x)< u_n(x)<\varphi_{2,n}(x)\},\\
&\Omega^2_n:=\{x\in\Omega_n\,:\, u_n(x)=\varphi_{1,n}(x)\},\\
&\Omega^3_n:=\{x\in\Omega_n\,:\, u_n(x)=\varphi_{2,n}(x)\},
\end{split}
\end{equation}
hence $\Omega_n=\Omega^1_n\cup\Omega^2_n\cup\Omega^3_n$.
\end{teo}

\proof The proof is completely analogous to the previous one.\endproof

\subsection{Uniqueness}

Thanks to the equivalent formulation \textit{(ii)} stated in Theorem \ref{Equi}, it is possible to prove the following uniqueness result for Problem \eqref{Pp}.
\begin{teo}\label{Uni}
Let us assume $2<sp<p$, $\mathcal{K}\neq\emptyset$ and that assumptions \eqref{GenHyp} and \eqref{HypOb} hold. Then Problem \eqref{Pp} admits a unique solution.
\end{teo}
\proof Thanks to the statement \textit{(ii)} in Theorem \ref{Equi}, we know that Problem \eqref{Pp} is equivalent to the  variational inequality \eqref{DVp}.\\
Let us consider $u_1, u_2\in\mathcal{K}$ solutions to Problem \eqref{DVp}. Hence, we rewrite \eqref{DVp}, first taking $u=u_1$ and $v=u_2$, and then taking $u=u_2$ and $v=u_1$. So, we get:
\begin{equation}\label{1}
\begin{aligned}
&\iint_{\Omega\times\Omega}\frac{|u_1(x)-u_1(y)|^{p-2}[u_1(x)-u_1(y)][w(x)-w(y)]}{|x-y|^{2+sp}}\dx\dy-\int_{\Omega}f(x)w(x)\dx\\
&+\int_{\partial\Omega}b(x)|u_1(x)|^{p-2}u_1(x)w(x)\dm\geq0, \text{ with } w(x)=u_2(x)-u_1(x),
\end{aligned}
\end{equation} 
and
\begin{equation}\label{2}
\begin{aligned}
&\iint_{\Omega\times\Omega}\frac{|u_2(x)-u_2(y)|^{p-2}[u_2(x)-u_2(y)][z(x)-z(y)]}{|x-y|^{2+sp}}\dx\dy-\int_{\Omega}f(x)z(x)\dx\\
&+\int_{\partial\Omega}b(x)|u_2(x)|^{p-2}u_2(x)z(x)\dm\geq0,  \text{ with } z(x)=u_1(x)-u_2(x),
\end{aligned}
\end{equation} 
Since $z(x)=-w(x)$, we can write \eqref{2} in the following form:
\begin{equation}\label{2bis}
\begin{aligned}
&-\iint_{\Omega\times\Omega}\frac{|u_2(x)-u_2(y)|^{p-2}[u_2(x)-u_2(y)][w(x)-w(y)]}{|x-y|^{2+sp}}\dx\dy+\int_{\Omega}f(x)w(x)\dx\\
&-\int_{\partial\Omega}b(x)|u_2(x)|^{p-2}u_2(x)w(x)\dm\geq0,  \text{ with } w(x)=u_2(x)-u_1(x).
\end{aligned}
\end{equation} 
Summing \eqref{1} ad \eqref{2bis} and changing the sign, we have:
\begin{equation}\label{3}
\begin{aligned}
&\iint_{\Omega\times\Omega}\frac{\{|u_2(x)-u_2(y)|^{p-2}[u_2(x)-u_2(y)]-|u_1(x)-u_1(y)|^{p-2}[u_1(x)-u_1(y)]\}[w(x)-w(y)]}{|x-y|^{2+sp}}\dx\dy\\
&+\int_{\partial\Omega}b(x)\left[|u_2(x)|^{p-2}u_2(x)-|u_1(x)|^{p-2}u_1(x)\right][u_2(x)-u_1(x)]\dm\leq0, \text{ with } w(x)=u_2(x)-u_1(x).
\end{aligned}
\end{equation} 
Now, applying the second inequality of Lemma 2.1 in \cite{LB} (with $\delta=p-2>0$), we obtain, for some constants $M_1,M_2>0$,
\begin{equation}\label{4}
M_1\iint_{\Omega\times\Omega}\frac{|[u_2(x)-u_1(x)]-[u_2(y)-u_1(y)]|^{p}}{|x-y|^{2+sp}}\dx\dy+M_2\int_{\partial\Omega}b(x)|u_1(x)-u_2(x)|^{p}\dm\leq0.
\end{equation} 
Then, we have
\begin{equation}\label{5}
C\left[|u_2(x)-u_1(x)|_{W^{s,p}(\Omega)}+\int_{\partial\Omega}b(x)|u_1(x)-u_2(x)|^{p}\dm\right]\leq0, \text{ with } 0<C=\min\{M_1,M_2\}.
\end{equation}
Finally, by Lemma \ref{EquiNorm}, we deduce that 
$$\|u_2-u_1\|_{b,s,p}\leq0\Longleftrightarrow\|u_2-u_1\|_{b,s,p}=0\Longleftrightarrow u_1=u_2.$$
\endproof

Thanks to Theorem \ref{EquiN}, the following analogous uniqueness result can be easily proved for Problem \eqref{Ppn}.
\begin{teo}\label{UniN}
Let us assume $2<sp<p$, $\mathcal{K}^n\neq\emptyset$ and that assumptions \eqref{GenHyp} and \eqref{HypObn} hold. Then Problem \eqref{Ppn} admits a unique solution.
\end{teo}
\proof The proof is completely analogous to the previous one.\endproof

\section{Asymptotic results}\label{AR}
\setcounter{section}{3} \setcounter{equation}{0}

In this section we investigate the asymptotic behavior of the solutions to the obstacle problems both as $p\to+\infty$ and $n\to+\infty$.\\
From now on we will use the following notations:
$$\|u\|_{L^p(\Omega)}=\|u\|_p \text{ and } \|u\|_{L^{\infty}(\Omega)}=\|u\|_{\infty},$$
$$\|u\|_{L^p(\Omega_n)}=\|u\|_{p,n}\text{ and } \|u\|_{L^{\infty}(\Omega_n)}=\|u\|_{\infty,n},$$
$$\|u\|_{W^{s,p}(\Omega)}=\|u\|_{s,p}\text{ and } \|u\|_{W^{s,p}(\Omega_n)}=\|u\|_{s,p,n},$$
$$|u|_{W^{s,p}(\Omega)}=|u|_{s,p}\text{ and } |u|_{W^{s,p}(\Omega_n)}=|u|_{s,p,n},$$
$$(u,v)_{2,s,p}=(u,v)_{s,p}.$$

%

\subsection{Asymptotics for $p\to+\infty$}

In this subsection we study the asymptotic behavior of the solutions to Problems \eqref{Pp} and \eqref{Ppn}, as $p\to+\infty$. To this aim, we suppose $\varphi_1,\varphi_2\in W^{s,p}(\Omega)\cap C(\overline{\Omega})$ and $\varphi_{1,n},\varphi_{2,n}\in W^{s,p}(\Omega)\cap C(\overline{\Omega})$ for every $n\in\N$. Moreover, we assume that
\begin{equation}\notag
\max\{\|\varphi_1\|_{C(\overline{\Omega})},\|\varphi_2\|_{C(\overline{\Omega})}\}\leq 1,
\end{equation}
and, for every $n\in\N$,
\begin{equation}\notag
\max\{\|\varphi_{1,n}\|_{C(\overline{\Omega})},\|\varphi_{2,n}\|_{C(\overline{\Omega})}\}\leq 1.
\end{equation}

Since it will be useful in the following, let us define the set
\begin{equation}\label{Kinfty}
\mathcal{K}_{\infty}=\{v\in W^{s,\infty}(\Omega)\,:\,\varphi_1\leq v\leq\varphi_2 \text{ in }\Omega \text{ with } \|v\|_{s,\infty}\leq1 \}.
\end{equation}
The following results hold.
\begin{teo}\label{Ap}
Let us assume $\mathcal{K}_{\infty}\neq\emptyset$. Let $f\in L^1(\Omega)$ and $\varphi_1,\varphi_2\in W^{s,p}(\Omega)\cap C(\overline{\Omega})$. Then, the set $\{u_p\,:\,u_p \text{ minimizer to Problem } \eqref{Pp}\}$, with $p$ such that $p>sp>2$, is precompact in $C(\overline{\Omega})$. Moreover $u_p\to u_{\infty}$ weakly in $W^{s,q}(\Omega)$, for every $q$ such that $q>sq>2$, with $u_{\infty}$ maximizer to Problem
$$\max\left\{\int_{\Omega}f(x)v(x)\dx\,:\,v\in\mathcal{K}_{\infty}\right\}.$$
\end{teo}

\proof Let $u_p\in\mathcal{K}$ be the solution to Problem \eqref{Pp}. To reach our goal, we need the evaluate the norm of $u_p$ in $W^{s,p}(\Omega)$.
Since $$\min_{\substack{\overline{\Omega}}}\varphi_1\leq \varphi_1\leq u_p\leq \varphi_2\leq\max_{\substack{\overline{\Omega}}}\varphi_2,$$
we deduce that, under the hypotheses on the obstacles,
$$\|u_p\|_p\leq |\Omega|^\frac{1}{p}\,\|u_p\|_{\infty}\leq |\Omega|^\frac{1}{p}.$$
Now, in order to complete the estimate, let us consider $v\in\mathcal{K}_{\infty}\subset\mathcal{K}$. We have that:
\begin{equation}\notag
\begin{split}
J_p(u_p)&\leq J_p(v)\Longleftrightarrow\frac{1}{p}\iint_{\Omega\times\Omega}\frac{|u_p(x)-u_p(y)|^p}{|x-y|^{2+sp}}\dx\dy-\int_{\Omega}f(x)u_p(x)\dx+\frac{1}{p}\int_{\partial\Omega}b(x)|u_p(x)|^{p}\dm\\[2mm]
&\leq\frac{1}{p}\iint_{\Omega\times\Omega}\frac{|v(x)-v(y)|^p}{|x-y|^{2+sp}}\dx\dy-\int_{\Omega}f(x)v(x)\dx+\frac{1}{p}\int_{\partial\Omega}b(x)|v(x)|^{p}\dm.
\end{split}
\end{equation}
This implies that
\begin{equation}\notag
\begin{split}
|u_p|^p_{s,p}&=\iint_{\Omega\times\Omega}\frac{|u_p(x)-u_p(y)|^p}{|x-y|^{2+sp}}\dx\dy\leq\iint_{\Omega\times\Omega}\frac{|v(x)-v(y)|^p}{|x-y|^{2+sp}}\dx\dy+p\int_{\Omega}f(x)[u_p(x)-v(x)]\dx\\[2mm]
&+\int_{\partial\Omega}b(x)\left[|v(x)|^{p}-|u_p(x)|^p\right]\dm\leq C\iint_{\Omega\times\Omega}\frac{1}{|x-y|^{2}}\dx\dy\\[2mm]
&+p\|f\|_1\|\varphi_2-\varphi_1\|_{\infty}+C\|b\|_{C(\overline\Omega)}\|v\|_{s,p}^p\leq\left[C_1+p\|f\|_1C_2+C_3\right],
\end{split}
\end{equation}
for some positive constants $C_1$, $C_2$ and $C_3$ independent from $p$. Then:
$$|u_p|_{s,p}\leq\left[C_1+p\|f\|_1C_2+C_3\right]^{\frac{1}{p}}\quad \forall\,p \text{ such that } p>sp>2.$$
So, considering $q$ such that $2<sq<q<p$, we have
$$|u_p|_{s,q}=\left(\iint_{\Omega\times\Omega}\frac{|u_p(x)-u_p(y)|^q}{|x-y|^{2+sq}}\dx\dy\right)^\frac{1}{q}\leq\left[\iint_{\Omega\times\Omega}\frac{|u_p(x)-u_p(y)|^p}{|x-y|^{2+sp}}\dx\dy\right]^{\frac{1}{p}}\cdot\left[ \iint_{\Omega\times\Omega}\frac{1}{|x-y|^2}\dx\dy\right]^{\frac{1}{q}-\frac{1}{p}}$$
$$\leq|u_p|_{s,p}C^{\frac{1}{q}-\frac{1}{p}}\leq\left[C_1+p\|f\|_1C_2+C_3\right]^{\frac{1}{p}}C^{\frac{1}{q}-\frac{1}{p}}.$$
Hence, we get:
$$\limsup_{p\to+\infty}|u_p|_{s,q}\leq C^{\frac{1}{q}}.$$
This, together with the hypotheses on the obstacles, implies that
$$\limsup_{p\to+\infty}\|u_p\|_{s,q}\leq C^{\frac{1}{q}},$$
with $C$ independent on $p$. Hence, there exists a subsequence, denoted by $u_{p_m}$, such that, for $m\to+\infty$,
$$u_{p_m}\to  u_{\infty} \text{ weakly in } W^{s,q}(\Omega) \text{ and uniformly in } C(\overline{\Omega}).$$
So, we have:
$$\|u_{\infty}\|_{s,\infty}=\lim_{q\to+\infty}\|u_{\infty}\|_{s,q}\leq\lim_{q\to+\infty}\liminf_{p\to+\infty}\|u_p\|_{s,q}\leq\lim_{q\to+\infty}\limsup_{p\to+\infty}\|u_p\|_{s,q}\leq\lim_{q\to+\infty}C^{\frac{1}{q}}=1.$$
And then, $u_{\infty}\in\mathcal{K}_{\infty}$.\\
Finally, again by the fact that $u_p$ solves Problem \eqref{Pp}, for every $v\in\mathcal{K}_{\infty}$, we have:
\begin{equation}\notag
\begin{split}
&-\int_{\Omega}f(x)u_p(x)\dx\leq\frac{1}{p}\iint_{\Omega\times\Omega}\frac{|u_p(x)-u_p(y)|^p}{|x-y|^{2+sp}}\dx\dy-\int_{\Omega}f(x)u_p(x)\dx+\frac{1}{p}\int_{\partial\Omega}b(x)|u_p(x)|^{p}\dm\\[2mm]
&\leq\frac{1}{p}\iint_{\Omega\times\Omega}\frac{|v(x)-v(y)|^p}{|x-y|^{2+sp}}\dx\dy-\int_{\Omega}f(x)v(x)\dx+\frac{1}{p}\int_{\partial\Omega}b(x)|v(x)|^{p}\dm\leq\frac{C}{p}-\int_{\Omega}f(x)v(x)\dx.
\end{split}
\end{equation}
So, considering the first and the last term in this chain on inequalities and passing to the limit as $p\to+\infty$ in them, we get
$$\lim_{p\to+\infty}\left(-\int_{\Omega}f(x)u_p(x)\dx\right)\leq\lim_{p\to+\infty}\left(\frac{C}{p}-\int_{\Omega}f(x)v(x)\dx\right),$$
that is
$$\int_{\Omega}f(x)u_{\infty}(x)\dx\geq\int_{\Omega}f(x)v(x)\dx, \forall v\in\mathcal{K}_{\infty}.$$
\endproof

We now consider the pre-fractal problems \eqref{Ppn}, for $n\in\N$ fixed. Coherently with the fractal case, we set
\begin{equation}\label{Kinftyn}
\mathcal{K}^n_{\infty}=\{v\in W^{s,\infty}(\Omega_n)\,:\,\varphi_{1,n}\leq v\leq\varphi_{2,n} \text{ in }\Omega_n \text{ with } \|v\|_{s,\infty,n}\leq1\}.
\end{equation}

\begin{teo}\label{Apn}
Let us assume $\mathcal{K}^n_{\infty}\neq\emptyset$. Let $f\in L^1(\Omega)$ and $\varphi_{1,n},\varphi_{2,n}\in W^{s,p}(\Omega)\cap C(\overline{\Omega})$ for every $n\in\N$. Then, the set $\{u_{p,n}\,:\,u_{p,n} \text{ minimizer to Problem } \eqref{Ppn}\}$, with $p$ such that $p>sp>2$, is precompact in $C(\overline{\Omega}_n)$. Moreover $u_{p,n}\to u_{\infty,n}$ weakly in $W^{s,q}(\Omega_n)$, for every $q$ such that $q>sq>2$, with $u_{\infty,n}$ maximizer to Problem
$$\max\left\{\int_{\Omega_n}f(x)v(x)\dx\,:\,v\in\mathcal{K}^n_{\infty}\right\}.$$
\end{teo}
\proof 
The proof is completely analogous to the one of Theorem \ref{Ap}.
\endproof

\subsection{Asymptotics for $n\to+\infty$ and $p<\infty$}

Along with the problems in the fractal boundary domain $\Omega$, we can also consider the analogous problems in the pre-fractal approximating domains $\Omega_n$.\\ 
In particular, let $n$ be a fixed positive integer and $\delta_n:=\left(\frac{3}{4}\right)^n$. Given $f\in L^1(\Omega)$, we recall the pre-fractal obstacle problems stated in \eqref{DVpn}: for every $n\in\N$,
\begin{equation}\notag
\text{find }  u_n\in\mathcal{K}^n\,\,:\,\,a_{p,n}(u_n,v_n-u_n)-\int_{\Omega_n}f(x)\left(v_n(x)-u_n(x)\right)\dx\geq 0\quad \forall\,v_n\in\mathcal{K}^n.
\end{equation} 

\bigskip

Analogously to before, we consider the problems
\begin{equation}\label{Pn}
\int_{\Omega_n}u_{\infty,n}(x)f(x)\dx=\max\left\{\int_{\Omega_n}w(x)f(x)\dx\,:\,w\in\mathcal{K}^n_{\infty}\right\}
\end{equation}
and
\begin{equation}\label{P}
\int_{\Omega}u_{\infty}(x)f(x)\dx=\max\left\{\int_{\Omega}w(x)f(x)\dx\,:\,w\in\mathcal{K}_{\infty}\right\}.
\end{equation}

We will need the following convergence result for the boundary term in \eqref{Jpn}. The proof follows at once by adapting the one of Proposition 2.5 in \cite{JCA}.

\begin{prop}\label{prop JCA}
Let $v_n\rightharpoonup u$ in $W^{s,p}(\Omega)$ and $b\in C(\overline\Omega)$. Then $$\displaystyle\delta_n\int_{\partial\Omega_n} b |v_n|^p\,\de\ell \xrightarrow[n\to+\infty]{}\int_{\partial\Omega} b |u|^p\,\de\mu.$$
\end{prop}

\begin{teo}\label{convergenza n p finito}
Let us assume $f\in L^1(\Omega)$, $\varphi_i,\varphi_{i,n}\in W^{s,p}(\Omega)$, $i=1,2$, such that $\varphi_1\leq\varphi_2$ on $\Omega$ and $\varphi_{1,n}\leq\varphi_{2,n}$ in $\Omega_n$ for every $n\in\N$. Suppose also that
\begin{equation}\notag
\varphi_{i,n}\xrightarrow[n\to+\infty]{}\varphi_i\quad\text{in}\quad W^{s,p}(\Omega),\quad i=1,2.
\end{equation}
If $u_{p,n}\in\mathcal{K}^n$ is a solution of \eqref{DVpn}, let $\hat{u}_{p,n}:=(\estdom\,u_{p,n})|_\Omega\in W^{s,p}(\Omega)$, where $\estdom$ is the extension operator given by Theorem \ref{teo estensione}.
Then, there exists a subsequence of $\hat{u}_{p,n}$ (still denoted with $\hat{u}_{p,n}$) which converges strongly in $W^{s,p}(\Omega)$ to a solution $u_p$ of \eqref{DVp}.
\end{teo}

\begin{proof}
By proceeding as in the proof of Theorem \ref{Ap}, we can prove that
\begin{equation}\label{equibdd}
\|u_{p,n}\|_{\infty,n}\leq C\text{ and }|u_{p,n}|^p_{s,p,n}\leq pC,
\end{equation}
where $C>0$ is independent from $n$. This implies that $u_{p,n}$ is bounded in $W^{s,p}(\Omega_n)$, for $2<sp<p$. Moreover, from the definition of $\hat{u}_{p,n}$, it holds that
\begin{equation}\notag
\|\hat{u}_{p,n}\|_{s,p}\leq\overline C_s\|u_{p,n}\|_{s,p,n},
\end{equation}
where $\overline C_s$ is the positive constant given by Theorem \ref{teo estensione} which is independent from $n$. Hence, $\hat{u}_{p,n}$ is equibounded in $W^{s,p}(\Omega)$, so it admits a subsequence, which we still denote by $\hat{u}_{p,n}$, which converges weakly to some $\hat{u}\in W^{s,p}(\Omega)$.

We recall that the solution $u_{p, n}$ to problem \eqref{DVpn} realizes the minimum on $\mathcal{K}^{n}$ of the functional $J_{p, n}(\cdot)$ defined in \eqref{Jpn}.

Our aim is to prove that
\begin{equation}\label{min hat u}
J_{p}(\hat{u})=\min_{v \in \mathcal{K}} J_{p}(v),
\end{equation}
where $J_{p}(\cdot)$ is the functional defined in \eqref{Jp}.\\
Indeed, since $\hat{u}_{p, n}$ weakly converges to $\hat{u}$ in $W^{s,p}(\Omega)$, we have that, for every fixed $m \in \mathbb{N}$,
\begin{equation}\label{stima 1}
\liminf_{n\to+\infty} |u_{p,n}|^p_{s,p,n}\geq \liminf_{n \to+\infty} |\hat{u}_{p,n}|^p_{s,p,m}\geq|\hat{u}|^p_{s,p,m}=\iint_{\Omega_m\times\Omega_m}\frac{|\hat{u}(x)-\hat{u}(y)|^p}{|x-y|^{2+sp}}\dx\dy.
\end{equation}

Moreover, since $\hat{u}_{p,n}$ coincides with $u_{p,n}$ on $\Omega_n$, from Proposition \ref{prop JCA} we have that $$\displaystyle\delta_n\int_{\partial\Omega_n} b |u_{p,n}|^p\,\de\ell \xrightarrow[n\to+\infty]{}\int_{\partial\Omega} b |\hat{u}|^p\,\de\mu.$$

Hence, passing to the limit for $m \to +\infty$ in \eqref{stima 1}, we obtain
\begin{equation}\label{stima 2}
J_{p}(\hat{u}) \leq \liminf_{n \to+\infty} J_{p,n}(u_{p, n})\leq \liminf_{n \to+ \infty} \min_{v \in \mathcal{K}^n} J_{p,n}(v).
\end{equation}

Moreover, given $u_p$ solution of problem \eqref{DVp}, we can construct a sequence of functions $w_{n} \in \mathcal{K}^{n}$ which strongly converges to $u_{p}$ in $W^{s,p}(\Omega)$ by setting
$$
w_{n}:=\varphi_{2, n} \wedge\left(u_{p} \vee \varphi_{1, n}\right),
$$
where we recall that $u \wedge v=\inf\{u,v\}$, $u \vee v=\sup\{u,v\}$ and $u^{+}=u \vee 0$.\\
We have that
$$
w_{n}=u_{p}+\left(\varphi_{1, n}-u_{p}\right)^{+}-\left(u_{p}+\left(\varphi_{1, n}-u_{p}\right)^{+}-\varphi_{2, n}\right)^{+}
$$
and, since $W^{s,p}(\Omega)$ is a Banach lattice,
\begin{equation}\label{stima 3}
\liminf_{n \to+\infty} \min_{v \in \mathcal{K}^{n}} J_{p,n}(v) \leq \liminf_{n \to+\infty} J_{p,n}(w_{n})=J_{p}(u_{p}).
\end{equation}
From \eqref{stima 2} and \eqref{stima 3}, we obtain \eqref{min hat u}.

Now, we will prove that the convergence is strong in $W^{s,p}(\Omega)$.\\
We remark that the following inequality holds, for every $p\geq 2$ and every $\xi,\eta\in\R^N$:
\begin{equation}\label{stimap}
\left(|\xi|^{p-2} \xi-|\eta|^{p-2} \eta, \xi-\eta\right)_{\mathbb{R}^{N}} \geq c_p|\xi-\eta|^{p},
\end{equation}
where $c_p\in(0,1]$ is a suitable constant depending on $p$ (for more details see \cite{biegert}).

Setting $v_{n}:=(\hat{u} \vee \varphi_{1, n}) \wedge \varphi_{2, n}$, it holds that $v_{n} \in \mathcal{K}_{n}$ and $v_{n} \rightarrow \hat{u}$ in $W^{s,p}(\Omega)$ as $n\to+\infty$. Then, for every $m \in \mathbb{N}$ such that $n \geq m$, we have that\\
\begin{equation}\notag
\begin{split}
c_{p} &|\hat{u}_{p,n}-\hat{u}|^p_{s,p,m}\leq c_{p}|\hat{u}_{p,n}-\hat{u}|^p_{s,p,n}\\
\leq&\iint_{\Omega_n\times\Omega_n}\left[\frac{|u_{p,n}(x)-u_{p,n}(y)|^{p-2}(u_{p,n}(x)-u_{p,n}(y))(u_{p,n}(x)-u_{p,n}(y)-(\hat{u}(x)-\hat{u}(y)))}{|x-y|^{2+sp}}\right.\\[2mm]
-&\left.\frac{|\hat{u}(x)-\hat{u}(y)|^{p-2}(\hat{u}(x)-\hat{u}(y))(u_{p,n}(x)-u_{p,n}(y)-(\hat{u}(x)-\hat{u}(y)))}{|x-y|^{2+sp}}\right]\dx\dy=(u_{p,n}, u_{p, n}-v_{n})_{s,p}\\[2mm]
+&(u_{p, n}, v_{n}-\hat{u})_{s,p}+\iint_{\Omega_n\times\Omega_n}\frac{|\hat{u}(x)-\hat{u}(y)|^{p-2}(\hat{u}(x)-\hat{u}(y))(u_{p,n}(x)-u_{p,n}(y)-(\hat{u}(x)-\hat{u}(y)))}{|x-y|^{2+sp}}\dx\dy
\end{split}
\end{equation}
\begin{equation}\label{ultima stima}
\begin{split}
\leq &\int_{\Omega_n} f\left(u_{p, n}-v_{n}\right)\dx-\frac{\delta_n}{p}\int_{\partial\Omega_n} b|u_{p,n}|^{p-2}u_{p,n}(u_{p,n}-v_n)\,\de\ell +(u_{p, n}, v_{n}-\hat{u})_{s,p}\\[2mm]
+&\iint_{\Omega_n\times\Omega_n}\frac{|\hat{u}(x)-\hat{u}(y)|^{p-2}(\hat{u}(x)-\hat{u}(y))(u_{p,n}(x)-u_{p,n}(y)-(\hat{u}(x)-\hat{u}(y)))}{|x-y|^{2+sp}}\dx\dy,
\end{split}
\end{equation}
where we used \eqref{stimap} and the fact that $u_{p,n}$ solves \eqref{DVpn}.

Passing to the $\limsup$ as $n \to+ \infty$ in the first and the last members of this chain of equalities and inequalities, we obtain that
\begin{equation}\label{limsup su m}
\limsup_{n \to+\infty} |\hat{u}_{p,n}-\hat{u}|^p_{s,p,m}= \limsup_{n \to+\infty}\iint_{\Omega_m\times\Omega_m}\frac{|\hat{u}_{p,n}(x)-\hat{u}_{p,n}(y)-(\hat{u}(x)-\hat{u}(y))|^{p}}{|x-y|^{2+sp}}\dx\dy\leq 0.
\end{equation}

Indeed, the fourth term in the right-hand side of \eqref{ultima stima} goes to zero since $\hat{u}_{p,n}$ weak converges to $\hat{u}$ in $W^{s,p}(\Omega)$. As to the first one, it holds that
$$
\int_{\Omega_n} f\left(u_{p, n}-v_{n}\right)\dx=\int_{\Omega} f\left(\hat{u}_{p, n}-v_{n}\right)\dx-\int_{\Omega\setminus\Omega_n} f\left(\hat{u}_{p, n}-v_{n}\right)\dx \xrightarrow[n\to+\infty]{} 0,
$$
since, when $n\to+\infty$, $\hat{u}_{p, n}$ and $v_{n}$ both strongly converge to $\hat{u}$ in $L^{p}(\Omega)$ and $\left|\Omega \setminus \Omega_n\right| \rightarrow 0$. 
The third term in \eqref{ultima stima} tends to zero because $v_{n}$ strongly converges to $\hat{u}$ in $W^{s,p}(\Omega)$.\\
It remains to study the second term in \eqref{ultima stima}. By using \eqref{ausiliare} and the trace theorem, we have that:
\begin{equation}\notag
\begin{split}
&-\frac{\delta_n}{p}\int_{\partial\Omega_n} b|u_{p,n}|^{p-2}u_{p,n}(u_{p,n}-v_n)\,\de\ell=\frac{\delta_n}{p}\int_{\partial\Omega_n} b|u_{p,n}|^{p-2}u_{p,n}(v_n-u_{p,n})\,\de\ell\leq\frac{\delta_n}{p^2}\int_{\partial\Omega_n}b\left(|v_n|^p-|u_{p,n}|^{p}\right)\,\de\ell\\
&\leq\frac{\delta_n}{p^2}\int_{\partial\Omega_n}b|v_n|^p\,\de\ell\leq\frac{\delta_n\|b\|_{C(\overline\Omega)}}{p^2}\int_{\partial\Omega_n}|v_n|^p\,\de\ell\leq\frac{\|b\|_{C(\overline\Omega)}C_{\Ext}}{p^2}\delta_n\|v_n\|^p_{s,p,\Omega},
\end{split}
\end{equation}
and the last quantity tends to zero as $n\to+\infty$ since $v_n$ strongly converges to $\hat{u}$ in $W^{s,p}(\Omega)$ and $\delta_n\to 0$.

Finally, passing to the limit for $m \to+\infty$ in \eqref{limsup su m}, we obtain the thesis.

\end{proof}

\subsection{Asymptotics for $p=\infty$ and $n\to+\infty$}

Finally, we perform the asymptotic analysis for $n \rightarrow+\infty$ when $p=\infty$. In this subsection, we suppose that the obstacles $\varphi_{1,n}$ and $\varphi_{2,n}$ belong to $\mathcal{K}^n_\infty$ for every $n\in\N$. Moreover, we replace the boundedness hypotheses in the convex sets $\mathcal{K}_\infty$ and $\mathcal{K}^n_\infty$ with the following: there exist two positive constants $C_1$ and $C_2$ such that
\begin{equation}\notag
\|v\|_{L^\infty(\Omega)}\leq C_1\quad\text{and}\quad |v|_{0,s}\leq C_2\quad\text{with } C_1+C_2\leq1.
\end{equation}
Here we denoted with $|\cdot|_{0,s}$ the seminorm of $C^{0,s}(\Omega)$.\\
Let $u_{\infty, n}$ be a maximizer of the problem \eqref{Pn}. For every $x\in\Omega$, we define
\begin{equation}\label{defmcsh}
\tilde{u}_{\infty, n}(x)=\sup_{y \in \Omega_n}\left\{u_{\infty, n}(y)-C_2|x-y|^s\right\}.
\end{equation}
From the definition of $\tilde{u}_{\infty,n}$ and from Corollary 2 in \cite{mcshane}, it follows that
\begin{equation}\label{bdd uinfty}
\left\|\tilde{u}_{\infty, n}\right\|_{W^{s,\infty}(\Omega)} \leq 1.
\end{equation}

\begin{teo}\label{convergenza n p infinito}
 Let $f \in L^{1}(\Omega)$ and the assumptions on $\varphi_1,\varphi_2,\varphi_{1,n},\varphi_{2,n}$ hold. Moreover, let us suppose that
\begin{equation}\label{conv ostacoli}
\varphi_{i,n}\xrightarrow[n\to+\infty]{}\varphi_i\quad\text{in}\quad C(\overline{\Omega}),\quad i=1,2.
\end{equation}
Then there exists a subsequence of $\tilde{u}_{\infty,n}$ defined in \eqref{defmcsh} such that $\tilde{u}_{\infty, n}$ $\star$-weakly converges as $n\to+\infty$ in $W^{s,\infty}(\Omega)$ to a maximizer $\tilde{u}_\infty$ of Problem \eqref{P}.
\end{teo}

\begin{proof} We adapt to our setting the proof of Theorem 4.2 in \cite{CF}.
Let $u_{\infty,n}$ be a maximizer of the problem \eqref{Pn}. By \eqref{bdd uinfty} we deduce that there exists $\tilde{u}_\infty \in W^{s,\infty}(\Omega)$ and a subsequence of $\tilde{u}_{\infty, n}$, which we still denote by $\tilde{u}_{\infty, n}$, which is $\star$-weakly converging to $\tilde{u}_\infty$ in $W^{s, \infty}(\Omega)$. Moreover,
$$
\|\tilde{u}_\infty\|_{W^{s,\infty}(\Omega)}\leq 1.
$$
We now construct, for every $w\in\mathcal{K}_{\infty}$, a sequence $w_{n}\in \mathcal{K}^n_\infty$ such that
$$
\lim_{n\to+\infty} \int_{\Omega_n} f w_{n}\,\dx=\int_{\Omega} f w\,\dx.
$$
First we define, for every $x \in \Omega$,
\begin{equation}\label{funzioni ausiliari}
\begin{split}
\tilde\varphi_{1, n}(x) & =\sup_{y \in \Omega}\left\{\varphi_{1, n}(y)-C_2|x-y|^s\right\},\\
\tilde\varphi_{2, n}(x) & =\inf_{y \in \Omega}\left\{\varphi_{2, n}(y)+C_2|x-y|^s\right\},\\
\tilde\varphi_{1}(x) & =\sup_{y \in \Omega}\left\{\varphi_{1}(y)-C_2|x-y|^s\right\},\\
\tilde\varphi_{2}(x) & =\inf_{y \in \Omega}\left\{\varphi_{2}(y)+C_2|x-y|^s\right\}.
\end{split}
\end{equation}
From direct calculations, for every $x \in \Omega$ it holds that
\begin{equation}\label{relazioniphi}
\varphi_{1, n}(x) \leq \tilde\varphi_{1, n}(x) \leq \tilde\varphi_{2, n}(x) \leq \varphi_{2, n}(x).
\end{equation}
Moreover, again from Corollary 2 in \cite{mcshane} and from direct inspection, we have that
\begin{equation}\label{bdd estese}
\left\|\tilde\varphi_{1, n}\right\|_{W^{s,\infty}(\Omega)} \leq 1 \quad\text{and}\quad\left\|\tilde\varphi_{2, n}\right\|_{W^{s,\infty}(\Omega)} \leq 1.
\end{equation}
For every $w \in\mathcal{K}_{\infty}$, we define $w_{n}:=\tilde\varphi_{2, n}\wedge\left(w \vee\tilde\varphi_{1, n}\right)$. From \eqref{relazioniphi} and \eqref{bdd estese} we have $w_n\in \mathcal{K}_{n}^{\infty}$. Moreover, from \eqref{conv ostacoli}, it follows that $w_{n} \rightarrow \tilde\varphi_{2} \wedge\left(w \vee\tilde\varphi_{1}\right)$ in $L^{\infty}(\Omega)$ as $n\to+\infty$.\\
Now, since $w \in \mathcal{K}_{\infty}$, hence in particular $\varphi_{1} \leq w \leq \varphi_{2}$, by direct calculation we deduce that $\tilde\varphi_{1}\leq w \leq\tilde\varphi_{2}$. Therefore, $w_{n} \rightarrow w$ in $L^{\infty}(\Omega)$ as $n\to+\infty$ and in particular, since $f\in L^1(\Omega)$,
$$
\lim_{n \to+\infty} \int_{\Omega_n} f w_{n}\,\dx=\int_{\Omega} f w\,\dx.
$$
Then, since $u_{\infty, n}$ is a maximizer of Problem \eqref{Pn}, e.g. in particular
$$
\int_{\Omega_n} f u_{\infty, n}\,\dx \geq \int_{\Omega_n} f w_{n}\,\dx,
$$
passing to the limit as $n\to+\infty$ we have
$$
\int_{\Omega} f \tilde{u}_\infty\,\dx \geq \int_{\Omega}f w\,\dx
$$
for every $w \in\mathcal{K}_{\infty}$. Hence, it follows that $\tilde{u}_\infty$ is a maximizer of Problem \eqref{P}, thus concluding the proof.
\end{proof}

\begin{rem}
We point out that the results of this paper can be generalized to more general irregular structures. For example, one can consider the case of a two-dimensional domain $\Omega_\alpha\subset\R^2$ having as boundary a (possibly random and/or asymmetric) Koch-type fractal mixture $K_\alpha$, for $\alpha\in (2,4)$ (see e.g. \cite{LaReVe,CF,CreoCSF}), along with the natural approximating domains $\Omega_\alpha^n$ having pre-fractal boundaries $K_\alpha^n$, for $n\in\N$. Moreover, under additional structural assumptions, these results could be applied also to the case of a $(\epsilon,\delta)$ domain (see \cite{Jones}) having as boundary a more general $d$-set.
\end{rem}

\section*{Statements and declarations} The authors have no competing interests to declare that are relevant to the content of this article. This research received no external funding.\\


\begin{thebibliography}{99}
\parskip0pt

\bibitem{AdHei} Adams, D.R., Hedberg, L.I.: Function Spaces and Potential Theory. Springer-Verlag, Berlin (1996).

\bibitem{AHW} Antil, H., Horton, M.O., Warma, M.: Exterior nonlocal variational inequalities associated with the fractional Laplace operator. Math. Control Relat. Fields (2025). https://doi.org/10.3934/mcrf.2025013


\bibitem{BFRO} Barrios, B., Figalli, A., Ros-Oton, X.: Global regularity for the free boundary in the obstacle problem for the fractional Laplacian. Amer. J. Math. 140, 415--447 (2018).

\bibitem{BDM} Bhattacharya, T., DiBenedetto, E., Manfredi, J.: Limits as $p\to+\infty$ of $\Delta_p u_p=f$ and related extremal problems. Some topics in nonlinear PDEs (Turin, 1989). Rend. Sem. Mat. Univ. Politec. Torino 1989, Special Issue, 15--68 (1991).

\bibitem{biegert} Biegert, M.: A priori estimate for the difference of solutions to quasi-linear elliptic equations. Manuscripta Math. 133, 273--306 (2010).


\bibitem{BDD} Buccheri, S., Da Silva, J.V., De Miranda, L.H.: A system of local/nonlocal $p$-Laplacians: the eigenvalue problem and its asymptotic limit $p\to\infty$. Asymptot. Anal. 128, 149--181 (2022).

\bibitem{CSS} Caffarelli, L.A., Salsa, S., Silvestre, L.: Regularity estimates for the solution and the free boundary of the obstacle problem for the fractional Laplacian. Invent. Math. 171, 425--461 (2008). 

\bibitem{CF} Capitanelli, R., Fragapane, S.: Asymptotics for quasilinear obstacle problems in bad domains. Discrete Contin. Dyn. Syst. Ser. S 12, 43--56 (2019).

\bibitem{CF2} Capitanelli, R., Fragapane, S.: Convergence results for the solutions of $(p,q)$-Laplacian double obstacle problems on  irregular domains. Submitted. https://doi.org/10.48550/arXiv.2312.16574.

\bibitem{CFV} Capitanelli, R., Fragapane, S., Vivaldi, M.A.: Regularity results for $p$-Laplacians in pre-fractal domains. Adv. Nonlinear Anal. 8, 1043--1056 (2019).

\bibitem{CV1} Capitanelli, R., Vivaldi, M.A.: FEM for quasilinear obstacle problems in bad domains. ESAIM Math. Model. Numer. Anal. 51, 2465-2485 (2017).

\bibitem{CreoCSF} Cefalo, M., Creo, S., Lancia, M.R., Rodr\'iguez-Cuadrado, J.: Fractal mixtures for optimal heat draining. Chaos Solitons Fractals 173, 11 pp. (2023).

\bibitem{CDL2012} Cefalo, M., Dell'Acqua, G., Lancia, M.R.: Numerical approximation of transmission problems across Koch-type highly conductive layers. Appl. Math. Comput. 218, 5453--5473 (2012).

\bibitem{CLM} Chambolle, A., Lindgren, E., Monneau, R.: A H\"older infinity Laplacian. ESAIM Control Optim. Calc. Var. 18, 799--835 (2012).

\bibitem{CLNODEA} Creo, S., Lancia, M.R.: Fractional $(s,p)$-Robin-Venttsel' problems on extension domains. NoDEA Nonlinear Differential Equations Appl. 28, 33 pp (2021).

\bibitem{CLADE} Creo, S., Lancia, M.R.: Dynamic boundary conditions for time dependent fractional operators on extension domains. Adv. Differential Equations 29, 727--756 (2024).

\bibitem{CPAA} Creo, S. Lancia, M.R., V\'elez-Santiago, A., Vernole, P.: Approximation of a nonlinear fractal energy functional on varying Hilbert spaces. Commun. Pure Appl. Anal. 17, 647--669 (2018).

\bibitem{JEE} Creo, S., Lancia, M.R., Vernole, P.: Convergence of fractional diffusion processes in extension domains. J. Evol. Equ. 20, 109--139 (2020).

\bibitem{JCA} Creo, S., Lancia M.R., Vernole, P.: M-convergence of $p$-fractional energies in irregular domains. J. Convex Anal. 28, 509--534 (2021).

\bibitem{dTEL} del Teso, F., Endal, J., Lewicka, M.: On asymptotic expansions for the fractional infinity Laplacian. Asymptot. Anal. 127, 201--216 (2022).

\bibitem{DIA} Diaz, J.I.: Nonlinear Partial Differential Equations and Free Boundaries. Vol. I: Elliptic equations. Pitman, Boston (1985).

\bibitem{hitch} Di Nezza, E., Palatucci, P., Valdinoci, E.: Hitchhiker's guide to the fractional Sobolev spaces. Bull. Sci. Math. 136, 521--573 (2012).


\bibitem{falconer} Falconer, K.: The Geometry of Fractal Sets. Cambridge Univ. Press, Cambridge (1990).

\bibitem{FP} Ferreira, R., P\'erez-Llanos, M.: Limit problems for a fractional $p$-Laplacian as $p\to\infty$. NoDEA Nonlinear Differential Equations Appl. (2016). https://doi.org/10.1007/s00030-016-0368-z

\bibitem{F} Fragapane, S.: $\infty$-Laplacian obstacle problems in fractal domains. in: Lancia, M.R., Rozanova-Pierrat A. (eds.) Fractals in engineering: theoretical aspects and numerical approximations, pp. 55--77. SEMA SIMAI Springer Ser. 8, Springer, Cham (2021).

\bibitem{giglimosconi} Gigli, N., Mosconi, S.: The abstract Lewy-Stampacchia inequality and applications. J. Math. Pures Appl. 104, 258--275 (2015).

\bibitem{HU} Hutchinson, J.E.: Fractals and selfsimilarity. Indiana Univ. Math. J. 30, 713--747 (1981).  

\bibitem{IMS} Iannizzotto, A., Mosconi, S., Squassina, M.: Fine boundary regularity for the degenerate fractional $p$-Laplacian. J. Funct. Anal. 279, 54 pp. (2020).

\bibitem{Jones} Jones, P.W.: Quasiconformal mapping and extendability of functions in Sobolev spaces. Acta Math. 147, 71--88 (1981).

\bibitem{JoWa} Jonsson, A., Wallin, H.: Function Spaces on Subsets of $\mathbb{R}^n$. Part 1. Harwood Acad. Publ., London (1984).

\bibitem{JoWa2} Jonsson, A., Wallin, H.: The dual of Besov spaces on fractals. Studia Math. 112, 285--300 (1995).

\bibitem{LaReVe} Lancia, M.R., Regis Durante, V., Vernole, P.: Density results for energy spaces on some fractafolds. Z. Anal. Anwend. 34, 357--372 (2015).


\bibitem{LB} Liu, W.B., Barrett, J.W., Quasi-norm error bounds for finite element approximation of some degenerate quasilinear elliptic equations and variational inequalities. RAIRO Mod\'el. Math. Anal. Num\'er. 28, 725--744 (1994). 

\bibitem{MAN} Mandelbrot, B.B.: The Fractal Geometry of Nature. W.H. Freeman \& Co., San Francisco (1982).

\bibitem{mcshane} McShane, E.J.: Extension of range of functions. Bull. Amer. Math. Soc. 40, 837--842 (1934).

\bibitem{MRT} Maz\'on, J.M., Rossi, J.D., Toledo, J.: Mass transport problems for the Euclidean distance obtained as limits of $p$-Laplacian type problems with obstacles. J. Differential Equations 256, 3208--3244 (2014).

\bibitem{MMV} Moreno M\'erida, L., Vidal, R.E.: The obstacle problem for the infinity fractional Laplacian. Rend. Circ. Mat. Palermo II Ser. (2018). https://doi.org/10.1007/s12215-016-0286-2

\bibitem{mosco1} Mosco, U.: Convergence of convex sets and solutions of variational inequalities. Adv. in Math. 3, 510--585 (1969).

\bibitem{MV} Mosco, U., Vivaldi, M.A.: Layered fractal fibers and potentials. J. Math. Pures Appl. 103, 1198--1227 (2015).

\bibitem{PW} Papageorgiou, N.S., Winkert, P.: Solutions with sign information for nonlinear nonhomogeneous problems. Math. Nachr. 292, 871--891 (2019).

\bibitem{rosoton} Ros-Oton, X.: Obstacle problems and free boundaries: an overview. SeMA J. 75, 399--419 (2018).


\bibitem{T} Troianiello, G.M.: Elliptic Differential Equations and Obstacle Problems. Plenum Press, New York (1987).

\bibitem{warmaCPAA} Warma, M.: A fractional Dirichlet-to-Neumann operator on bounded Lipschitz domains. Commun. Pure Appl. Anal. 14, 2043--2067 (2015).

\bibitem{warmaNODEA} Warma, M.: The fractional Neumann and Robin type boundary conditions for the regional fractional $p$-Laplacian. NoDEA Nonlinear Differential Equations Appl. 23, 46 pp. (2016).


\end{thebibliography}
\end{document}